\documentclass[12pt]{amsart}
\usepackage{amssymb,amsmath,amscd}
\usepackage{url}
\usepackage[all,cmtip]{xy}
\usepackage{graphicx}
\usepackage{tikz}
\usepackage{pgfplots}
\usepackage{enumerate}
\usetikzlibrary{matrix,arrows,decorations.pathmorphing}

\usepackage[mathlines,pagewise]{lineno}

\usepackage{graphicx}

\usepackage{ifthen, xcolor}
\newboolean{commentBoolVar}
\setboolean{commentBoolVar}{true} %SET THIS TO TRUE TO SHOW COMMENTS, FALSE TO HIDE THEM

\newcommand{\colourcomment}[3]
{
\ifthenelse{\boolean{commentBoolVar}}{{\color{#2}(#1: #3)}}{}
}

\newtheorem{theorem}{Theorem}[section]
\newtheorem{lemma}[theorem]{Lemma}
\newtheorem{proposition}[theorem]{Proposition}

\newtheorem{example}[theorem]{Example}
\newtheorem*{theorem*}{Theorem}
\newtheorem*{lemma*}{Lemma}
\newtheorem*{proposition*}{Proposition}
\newtheorem*{corollary*}{Corollary}

\newtheorem{remark}[theorem]{Remark}
\newtheorem{problem}[theorem]{Problem}

\newcommand{\C}{\mathbb C}
\newcommand{\Q}{\mathbb Q}
\newcommand{\Z}{\mathbb Z}

\newcommand{\be}{\begin{eqnarray}}
\newcommand{\ee}{\end{eqnarray}}

\newcommand{\End}{\mathrm{End}}
\newcommand{\Perm}{\mathrm{Perm}}
\newcommand{\Aut}{\mathrm{Aut}}
\newcommand{\Map}{\mathrm{Map}}
\newcommand{\F}{\mathbb{F}}
\newcommand{\Hol}{\mathrm{Hol}}

\begin{document}

\title[Hopf Forms and Hopf-Galois Theory]{Hopf Forms and Hopf-Galois Theory}
%\author{Alan Koch}
%\address{Department of Mathematics, Agnes Scott College, 141 E. College Ave., Decatur, GA 30030 USA}
%\email{akoch@agnesscott.edu}
\author{Timothy Kohl}
\address{Department of Mathematics and Statistics, Boston University, 111 Cummington Mall, Boston, MA 02215 USA}
\email{tkohl@math.bu.edu} 
%\author{Paul J.~Truman}
%\address{School of Computing and Mathematics, Keele University, Staffordshire, ST5 5BG, UK}
%\email{P.J.Truman@Keele.ac.uk}
\author{Robert Underwood}
\address{Department of Mathematics and Department of Computer Science, Auburn University at Montgomery, Montgomery, AL, 36124 USA}
\email{runderwo@aum.edu}

\date{\today}

\maketitle

%\begin{abstract}  
%\end{abstract}
%\noindent {\it key words:} group ring; Hopf form; cyclic group; Hopf-Galois structure \par
%\noindent {\it MSC:} \par

%\RU{This is an expansion/revision of previous work.  Thanks to Tim for discussions regarding papers \cite{HP86}, \cite{Pa90}.}

\begin{abstract}
Let $K$ be a finite field extension of $\Q$ and let $N$ be a finite group with automorphism group $F=\Aut(N)$.  R. Haggenm\"{u}ller and B. Pareigis have shown that there is a bijection 
\[\Theta: {\mathcal Gal}(K,F)\rightarrow {\mathcal Hopf}(K[N])\] 
from the collection of $F$-Galois extensions of $K$ to the collection of Hopf forms of the group ring $K[N]$.  
For $N=C_n$, $n\ge 1$, $C_p^m$, $p$ prime, $m\ge 1$, and $N=D_3,D_4,Q_8$, we show that $\Q[N]$ admits an absolutely semisimple Hopf form $H$
and find $L$ for which $\Theta (L)=H$.  Moreover, if $H$ is the Hopf algebra given by a Hopf-Galois structure on a Galois extension $E/K$, we show how to construct the preimage of $H$ under $\Theta$ assuming certain conditions.
\end{abstract}

\section{Introduction}
Let $K$ be a finite field extension of $\Q$ and let $N$ be a finite group with automorphism group $F=\Aut(N)$.  By a result of R. Haggenm\"{u}ller and B. Pareigis there is a bijection 
\[\Theta: {\mathcal Gal}(K,F)\rightarrow {\mathcal Hopf}(K[N])\] 
from the collection of $F$-Galois extensions of $K$ to the collection of Hopf forms of the group ring $K[N]$ \cite[Theorem 5]{HP86}.   This paper concerns the map $\Theta$ and its inverse $\Theta^{-1}$.

%as applied to certain $F$-Galois extensions and certain Hopf forms of $KN$.   

The map $\Theta$ can be used to classify Hopf forms of $K[N]$ for certain $N$. 
For instance let $N=C_n$ denote the cyclic group of order $n$.  If $n=2$ then $F$ is trivial and so $K$ is the only $F$-Galois extension of $K$.  Consequently $\Theta(K)=K[C_2]$ is the only Hopf form of 
$K[C_2]$.  For the cases $n=3,4,6$, 
\[F=\Aut(C_n)=\Z_n^*=C_2.\]
The $C_2$-Galois extensions of $K$ are completely classified as the quadratic extensions $L=K[x]/(x^2-b)$,
where $b\in K^\times$ \cite{Sm72}.  Thus the map $\Theta$ yields a complete classification of all Hopf 
forms of $K[C_n]$ for the cases $n=3,4,6$ \cite[Theorem 6]{HP86}.  

%In the cases $n\not = 2,3,4,6$, the $F$-Galois extensions of $K$ 
%(and consequently) the Hopf forms of $K[C_n]$ seem difficult to compute.  See \cite[pp. 88-89]{Pa90} for a treatment of the 
%case $K=\Q$, $n=5$ (where $F=C_4$).  

%So it is of interest to investigate the structure of Hopf forms of $K[C_n]$ for $n\ge 2$.  

For $n\ge 2$, two special Hopf forms of $K[C_n]$ can be identified. The first is the trivial Hopf form $K[C_n]$.  If $L$ is the trivial $F$-Galois extension $\mathrm{Map}(F,K)$ of $K$, then $\Theta(L)=K[C_n]$.   The second is $(K[C_n])^*$, the linear dual of 
$K[C_n]$, which is the ``absolutely semisimple" Hopf form of 
$K[C_n]$.   If $L=K[x]/(\Phi_n(x))$, where $\Phi_n(x)$ is the $n$th cyclotomic polynomial, then $L$ is a 
$\Z_n^*$-Galois extension of $K$ and $\Theta(L)=(K[C_n])^*$.  %In the case that $K=\Q$ and $n=p^m$, $p$ an odd prime, $m\ge 1$, we obtain an explicit description of the absolutely semisimple Hopf form $(\Q[C_{p^m}])^*$ of $\Q[C_{p^m}]$.   

In \cite[p. 91]{Pa90}, B. Pareigis posed the following problem:  for which finite groups $N$ does $K[N]$, $K=\Q$, admit an 
absolutely semisimple Hopf form?  As noted above $\Q[C_n]$ admits an absolutely semisimple Hopf form.   We show that $\Q[N]$ admits such a form for the abelian group $C_p^m$, $m\ge 1$, and non-abelian groups  $D_3$, $D_4$ and $Q_8$.  In each case we compute the $F$-Galois extension $L$ of $\Q$ for  which $\Theta(L)$ is the absolutely semisimple Hopf form.

There is a natural application of $\Theta$ to Hopf-Galois theory.   Let $(H,\cdot)$ be a Hopf-Galois structure of type $N$ 
on the Galois extension of fields $E/K$.  Then $H$ is a Hopf form of $K[N]$ and thus $\Theta(L)=H$ for some $F$-Galois extension $L$ of $K$, $F=\Aut(N)$.   We show how to construct the $F$-Galois extension $L$ for which 
$\Theta(L)=H$ under certain conditions.  

%We identify necessary conditions for the Galois extension $E/\Q$ with group $G$, $n=\vert G\vert$, 
%to admit the absolutely semisimple Hopf-Galois structure $((\Q[C_n])^*,\cdot)$ of type $C_n$.  

\section{Galois Extensions and Hopf Forms}

Let $F$ be a finite group.  An {\em $F$-Galois extension} of $K$ is a commutative $K$-algebra $L$ that satisfies

\vspace{.2cm}

(i)\ $F$ is a subgroup of $\mathrm{Aut}_K(L)$,

\vspace{.2cm}

(ii)\ $L$ is a finitely generated, projective $K$-module,

\vspace{.2cm}

(iii)\ $F\subseteq \mathrm{End}_K(L)$ is a free generating system over $K$. 

\vspace{.2cm}

\noindent The notion of $F$-Galois extension generalizes the usual definition of a Galois extension of fields.  The $K$-algebra of maps $\mathrm{Map}(F,K)$ is the {\em trivial $F$-Galois extension} of $K$ where the action of 
$F$ on $\mathrm{Map}(F,K)$ is given as $g(\phi)(h) = \phi(g^{-1}h)$ for $g,h\in F$, $\phi\in \Map(F,K)$.  We let ${\mathcal Gal}(K,F)$ denote the collection of all $F$-Galois extensions of $K$. 

B. Pareigis \cite[Theorem 4.2]{Pa90} has completely characterized $F$-Galois extensions.

\begin{theorem} [Pareigis] \label{galois} Let $F$ be a finite group.  Then $L$ is an $F$-Galois extension of $K$
if and only if 
\[L=\underbrace{M\times M\times \cdots \times M}_n\]
where $M$ is a $U$-Galois field extension of $K$ for some subgroup $U$ of $F$ of index $n$. 
\end{theorem}

\begin{proof} (Sketch)  Let $L$ be an $F$-Galois extension.  Then $L$ is a commutative, separable $K$-algebra and hence
\[L= M_1\times M_2\times \cdots \times M_n\]
where each $M_i$ is a separable field extension of $K$.  Let $f_1,f_2,\dots,f_n$ be the minimal orthogonal idempotents.
We have $M_i\cong M_j, \forall i,j$, hence
 \[L= \underbrace{M\times M\times \cdots \times M}_n,\]
where $M=M_1$.   Let $U$ be the stabilizer of $f_1$ in $F$.  Then $M$ is Galois over $K$ with group $U$ and $[L:U]=n$.

Conversely, let $U\le F$, $n=[F:U]$, and let $M/K$ be $U$-Galois. Let $g_1,g_2\cdots g_n$ be a left transversal for $U$ in $F$ and let $\rho: G\rightarrow S_n$ be defined as $\rho(g)(i)=j$ if $gg_iU=g_jU$.   Let 
\[L=\underbrace{M\times M\times \cdots \times M}_n\]
with minimal orthogonal idempotents $f_1,f_2,\cdots f_n$, and let the action of $F$ on $L$ be defined on each component as   
\[g(mf_i)=(g^{-1}_{\rho(g)(i)}gg_i)(m)f_{\rho(g)(i)},\]
for $m\in M$, $1\le i\le n$.   Then $L$ is an $F$-Galois extension of $K$.  (For details see \cite[Theorem 4.2]{Pa90}.)

 \end{proof}

Let $N$ be a finite group.  Then the group ring $K[N]$ is a $K$-Hopf algebra.  
Let $L$ be a faithfully flat $K$-algebra. An {\em $L$-Hopf form} of $K[N]$ is a $K$-Hopf algebra $H$ for which 
\[L\otimes_K H\cong L\otimes_K K[N]\cong L[N]\] 
as $L$-Hopf algebras.  A {\em Hopf form} of $K[N]$ is a $K$-Hopf algebra $H$ for which there exists a faithfully flat $K$-algebra $L$ with
$L\otimes_K H\cong L\otimes_K K[N]\cong L[N]$ as $L$-Hopf algebras.  The {\em trivial Hopf form} of $K[N]$ is $K[N]$.  
Let ${\mathcal Hopf}(K[N])$ denote the collection of all Hopf forms of $K[N]$.

R. Haggenm\"{u}ller and B. Pareigis \cite[Theorem 5]{HP86} have classified all Hopf forms of $K[N]$. 
 
\begin{theorem}  [Haggenm\"{u}ller and Pareigis]\label{HP}  Let $N$ be a finite group and let 
$F=\mathrm{Aut}(N)$.   There is a bijection $\Theta: {\mathcal Gal}(K,F)\rightarrow {\mathcal Hopf}(K[N])$ which associates to each $F$-Galois extension $L$ of $K$, the Hopf form 
$H=\Theta(L)$ of $K[N]$
defined as the fixed ring
\[H=(L[N])^F,\]
where the action of $F$ on $N$ is through the automorphism group $F=\mathrm{Aut}(N)$ and the action of $F$ on $L$
is the Galois action.  The Hopf form $H$ is an $L$-Hopf form of $K[N]$ with isomorphism 
$\psi: L\otimes_K H\rightarrow L[N]$ defined as $\psi(x\otimes h)=xh$.  
\end{theorem}

\begin{proposition} \label{trivial} Let $N$ be a finite group, let $F=\mathrm{Aut}(N)$, and let $L=\Map(F,K)$ denote the trivial $F$-Galois extension of $K$.   Then 
$\Theta(L)=(L[N])^{F}\cong K[N]$.
\end{proposition}

\begin{proof} (Sketch)  $H=(L[N])^{F}$ has a $K$-basis consisting of group-like elements.  Hence, 
$H=K[N']$ for some finite group $N'$.  Since $L\otimes_K H=L\otimes_K K[N']\cong L[N]$ as Hopf algebras, we conclude that $N'\cong N$. 
\end{proof}

We illustrate Proposition \ref{trivial} in the case that $N=C_3=\langle \sigma\rangle=\{1,\sigma,\sigma^2\}$, the cyclic group of order $3$.  Here, $F=\mathrm{Aut}(C_3)=C_2=\{1,g\}$ with action given as $1(\sigma)=\sigma$, $g(\sigma) = \sigma^2$.  A $K$-basis for $L=\Map(C_2,K)$ is $\{e_1,e_g\}$, the basis dual to the basis $\{1,g\}$ for $C_2$.  
We have $e_1(1)=1$, $e_1(g)=0$, $e_g(1)=0$, $e_g(g)=1$; $\{e_1,e_g\}$ is a basis of mutually orthogonal idempotents: $e_1^2=e_1$, $e_g^2=e_g$, $e_1e_g=e_ge_1=0$, $e_1+e_g=1$, thus,
\[L= Ke_1\oplus Ke_g.\]
The $C_2$-Galois action on $L$ is given as 
\[1\cdot e_1 = e_1,\quad 1\cdot e_g = e_g,\quad g\cdot e_1 = e_g,\quad g\cdot e_g = e_1.\]

We compute $(L[C_3])^{C_2}$ directly.  A typical element of $L[C_3]$ is $\sum_{i=0}^2(\alpha_ie_1+\beta_ie_g)\sigma^i$ for $\alpha_i,\beta_i\in K$.   We require that

\begin{eqnarray*}
g(\sum_{i=0}^2(\alpha_ie_1+\beta_ie_g)\sigma^i) & = &  \sum_{i=0}^2g(\alpha_ie_1+\beta_ie_g)\sigma^{2i} \\
& = &   \sum_{i=0}^2(\alpha_ie_g+\beta_ie_1)\sigma^{2i} \\
& = &  \sum_{i=0}^2(\alpha_ie_1+\beta_ie_g)\sigma^{i}.
\end{eqnarray*}
Thus $\alpha_0 = \beta_0$, $\alpha_1 = \beta_2$, $\alpha_2 = \beta_1$.  So a typical element of 
$(L[C_3])^{C_2}$ is 

\vspace{.2cm}

$(\alpha_0e_1+\beta_0e_g)1+ (\alpha_1e_1+\beta_1e_g)\sigma+ (\alpha_2e_1+\beta_2e_g)\sigma^2$

\vspace{-.5cm}

\begin{eqnarray*}
&= & \alpha_0(e_1+e_g)1+ (\alpha_1e_1+\alpha_2e_g)\sigma+ (\alpha_2e_1+\alpha_1e_g)\sigma^2 \\
& = & \alpha_01+\alpha_1(e_1\sigma+e_g\sigma^2)+\alpha_2(e_g\sigma+e_1\sigma^2).
\end{eqnarray*}
Thus a $K$-basis for $(L[C_3])^{C_2}$ is $\{1, e_1\sigma+e_g\sigma^2, e_g\sigma+e_1\sigma^2\}$.
We have 

\begin{eqnarray*}
\Delta(e_1\sigma+e_g\sigma^2) & = &  e_1(\sigma\otimes \sigma)+ e_g(\sigma^2\otimes \sigma^2)  \\
& = &  (e_1\sigma+e_g\sigma^2)\otimes (e_1\sigma+e_g\sigma^2),
\end{eqnarray*}
and so, $e_1\sigma+e_g\sigma^2$ is group-like.  Likewise $e_g\sigma+e_1\sigma^2$ is group-like. 
Thus $(L[C_3])^{C_2}\cong K[C_3]$. 

%\begin{remark}  \label{inverse} The illustration above shows that $\Map(C_2,K)=\Theta^{-1}(K[C_3])$.  In general, given a Hopf form $H$ of $K[N]$ it is not clear (at least to us) how to explicitly construct an element $L\in {\mathcal Gal}(K,F)$ for which $\Theta(L)=H$. 
%\end{remark}

\section{The Absolutely Semisimple Hopf Form of $K[C_n]$}

Let $N$ be any finite group.  For $n\ge 1$ let $\zeta_n$ denote a primitive $n$th root of unity.  By Maschke's theorem, $K[N]$ is semisimple.   Extending scalars to $\C$ yields the Wedderburn-Artin decomposition
\[\C[N]\cong \mathrm{Mat}_{n_1}(\C)\times \mathrm{Mat}_{n_2}(\C)\times\cdots\times 
\mathrm{Mat}_{n_l}(\C)\]
for integers $n_i\ge 1$, $1\le i\le l$.   Let $L$ be a faithfully flat $K$-algebra.  Then any $L$-Hopf form of $K[N]$ is also semisimple.  An $L$-Hopf form 
$H$ of $K[N]$ is {\em absolutely semisimple} if
\[H\cong \mathrm{Mat}_{n_1}(K)\times \mathrm{Mat}_{n_2}(K)\times\cdots\times 
\mathrm{Mat}_{n_l}(K).\]

\begin{proposition}  Let $N=C_n=\langle \sigma\rangle$, $n\ge 1$.  Then the linear dual $(K[C_n])^*$ is an absolutely semisimple Hopf form 
of $K[C_n]$.
\end{proposition}

\begin{proof}  It is well-known that $(K[C_n])^*$ is a $K$-Hopf algebra.   Let $\{p_j\}_{j=0}^{n-1}$ be the basis for $(K[C_n])^*$ dual to the basis $\{\sigma^i\}_{i=0}^{n-1}$ of $K[C_n]$.   Let $E/K$ be any field extension containing a primitive $n$th
root of unity.  Then $E$ is a faithfully flat $K$-algebra.   Let $\{\chi_j\}_{j=0}^{n-1}$ denote the set of $n$ irreducible characters of $C_n$, which all are of degree $1$ (since $C_n$ is abelian). The $\chi_j$ take values in $E$.  
There exists an isomorphism of $E$-Hopf algebras 
\[\varphi: E\otimes_K (K[C_n])^*\rightarrow E\otimes K[C_n]\cong E[C_n]\]
defined by
\[p_j\mapsto {1\over n} \sum_{i=0}^{n-1} \chi_j(\sigma^{-i})\sigma^i,\]
$0\le j\le n-1$ (cf. \cite[Exercise 6.4]{Se77}).
Thus $(K[C_n])^*$ is an $E$-form of $K[C_n]$. 
Since
\[\C[C_n]\cong \underbrace{\C\times \C\times \cdots \times \C}_n\]
and
\[(K[C_n])^*\cong \underbrace{K\times K\times \cdots \times K}_n,\]
$(K[C_n])^*$ is an absolutely semisimple Hopf form of $K[C_n]$. 

\end{proof}

B. Pareigis has shown that $(K[C_n])^*$ is the only absolutely semisimple Hopf form of $K[C_n]$ \cite[Theorem 4.3]{Pa90}.

\begin{theorem}[Pareigis]  \label{P} $K[C_n]$ has a uniquely determined absolutely semisimple Hopf form $H=(K[C_n])^*$, where $(K[C_n])^*$ is the linear dual of $K[C_n]$.  
\end{theorem}

As a Hopf form of $K[C_n]$, $(K[C_n])^*$ comes from some $F$-Galois extension $L$, that is, $\Theta(L)=(K[C_n])^*$ for some $L$.  In fact, if $L=K[x]/(\Phi_n(x))$, where $\Phi(x)$ is the $n$th cyclotomic polynomial, then $L$ is an $F$-Galois extension of $K$ and $\Theta(L)= (L[C_n])^F=(K[C_n])^*$.

Note that $K[x]/(\Phi_n(x))$ is not necessarily a field.  For example, if $n=15$ and $K=\Q(\zeta_3)$, then
\[K[x]/(\Phi_{15}(x))\cong K(\zeta_{15})\times K(\zeta_{15}).\]
In this case
\[F=\Aut(C_{15})=\Z_{15}^*=C_2\times C_4.\]  
The faithfully flat $K$-algebra 
$K(\zeta_{15})\times K(\zeta_{15})$ is a $(C_2\times C_4)$-Galois extension of $K$ with
\[\Theta(K(\zeta_{15})\times K(\zeta_{15}))=(K[C_{15}])^*.\]  

If $K=\Q$, then $\Q[x]/(\Phi_n(x))$ is a field, equal to $\Q(\zeta_n)$ 
where $\zeta_n$ denotes a primitive $n$th root of unity; $\Q(\zeta_n)$ is a $\Z_n^*$-Galois extension of $\Q$.     
In the case that $n=p^m$, for $p$ an odd prime, $m\ge 1$, the first author has shown that $(K[C_{p^m}])^*$ is the 
fixed ring \cite [Proposition 3.1]{Ko18}. 

\begin{proposition} [Kohl] \label{abss2} Let $p$ be an odd prime and let $L=\Q(\zeta_{p^m})$, $m\ge 1$.  Then
\[\Theta(L)= (L[C_{p^m}])^F=(\Q[C_{p^m}])^*,\]
where $F=\Aut(C_{p^m})=\Z_{p^m}^*$. 
\end{proposition}

\begin{proof}  Let $C_{p^m}=\langle \sigma\rangle$.  Then the elements 
\[e_j = {1\over p^m} \sum_{i=0}^{p^m-1}{\zeta_{p^m}^{-ij}}\sigma^i\in L[C_{p^m}],\]
$0\le j\le p^m-1$, are in the fixed ring $(L[C_{p^m}])^F$ and are mutually orthogonal idempotents.  
It follows that they form a $\Q$-basis for $(L[C_{p^m}])^F$.  Thus $(L[C_{p^m}])^F\cong (\Q[C_{p^m}])^*$
as $\Q$-Hopf algebras. 

\end{proof}

\section{A Problem of Pareigis}

B. Pareigis has stated the following problem in \cite[p. 91]{Pa90}. 

\begin{problem}[Pareigis] \label{qP} For which finite groups $N$ does $\Q[N]$ admit an absolutely semisimple Hopf form?
\end{problem}

We review solutions to Problem \ref{qP} for various $N$. 

\subsection{$N$ Abelian}

\begin{example}
Let $N=C_n$, $n\ge 1$.  By Theorem \ref{P}, $\Q[C_n]$ admits the absolutely semisimple Hopf form $(\Q[C_n])^*$.  
If $L=\Q(\zeta_n)$, then $\Theta(L)=(\Q[C_n])^*$. 
\end{example}

%\begin{example} $N=C_p\times C_p$, $p\ge 2$ prime.  Then
%$\Q(C_p\times C_p)$ admits the absolutely semisimple Hopf form
%\[(\Q(C_p\times C_p))^*\cong \underbrace{\Q\times \Q\times \cdots \times \Q}_{p^2}.\]
%We have 
%\[F=\mathrm{Aut}(C_p\times C_p)\cong \mathrm{GL}_2(\F_p).\]
%So there is an $F$-Galois extension $L/\Q$ for which
%\[\Theta(L)=(\Q(C_p\times C_p))^*.\]
%What is the structure of $L$?
%\end{example}

%\end{frame}

\begin{example} Let $N=C_p^m$ be the elementary abelian group of order $p^m$, $p\ge 2$, $m\ge 1$.    
Let $E=\Q(\zeta_p)$.  Then 
\[E\otimes_\Q (\Q[C_p^m])^*\cong E\otimes_\Q \Q[C_p^m],\]
as $E$-Hopf algebras.  Moreover,
\[\C[C_p^m]\cong \underbrace{\C\times \C\times \cdots \times \C}_{p^m}\]
and 
\[(\Q[C_p^m])^*\cong \underbrace{\Q\times \Q\times \cdots \times \Q}_{p^m}.\]
Thus
$\Q[C_p^m]$ admits the absolutely semisimple Hopf form $(\Q[C_p^m])^*$.
\end{example}

By Theorem \ref{HP} there is an $\mathrm{Aut}(C_p^m)$-Galois extension $L/\Q$ for which
\[\Theta(L)=(\Q[C_p^m])^*.\]
What is the structure of $L$?

We have $\mathrm{Aut}(C_p^m)=\mathrm{GL}_m(\F_p)$ with 
$\vert \mathrm{GL}_m(\F_p)\vert = \prod_{i=0}^{m-1} (p^m-p^i)$.
There is a subgroup of $\mathrm{GL}_m(\F_p)$ defined as
\[U = \{aI_m:\ a\in \F_p^*\}\cong \F_p^*\]
of index $l= (\prod_{i=0}^{m-1} (p^m-p^i))/(p-1)$; $\Q(\zeta_p)$ is a 
$U$-Galois extension of fields.  Thus by Theorem \ref{galois} 
there exists a $\mathrm{GL}_m(\F_p)$-Galois extension of $\Q$
\[L=\underbrace{\Q(\zeta_p)\times \Q(\zeta_p)\times \cdots \times \Q(\zeta_p)}_l.\]
Computing the fixed ring yields
\[\Theta(L)=(L[C_p^m])^{\mathrm{GL}_m(\F_p)}=(\Q[C_p^m])^*.\]

\subsection{$N$ Non-abelian}

We first prove a lemma.

\begin{lemma}  \label{aut} Let  $D_3$, $D_4$ denote the $3$rd and $4$th dihedral groups, respectively.  Then

\vspace{.2cm}

(i)\ $Aut(D_3)=D_3$,

\vspace{.2cm}

(ii)\ $Aut(D_4) = D_4$.

\end{lemma}

\begin{proof}  Let $D_n$ denote the $n$th order dihedral group.  Then $\Aut(D_n)\cong \Hol(C_n)$ since the cyclic
subgroup of $D_n$ is characteristic.  To prove (i) note that
\[\Hol(C_3) = C_3 \rtimes Aut(C_3)\cong C_3\rtimes C_2\cong D_3.\]
To prove (ii), observe that
\[\Hol(C_4) = C_4 \rtimes Aut(C_4)\cong C_4\rtimes C_2\cong D_4,\]
where $C_2$ acts on $C_4$ by inversion.

\end{proof}

\begin{example}
Let $N=D_3$.  We have
\[\C[D_3]\cong \C\times \C\times \mathrm{Mat}_{2}(\C)\]
and by \cite[Example (7.39)]{CR81}
\[\Q[D_3]\cong \Q\times \Q\times \mathrm{Mat}_{2}(\Q).\]
Thus $\Q[D_3]$ is an absolutely semisimple Hopf form of itself.  By Lemma \ref{aut}(i), $\Aut(D_3)=D_3$ and so
\[\Theta(\mathrm{Map}(D_3,\Q))=\Q[D_3].\]
\end{example}

\begin{example}
Let $N=D_4$.  We have
\[\C[D_4]\cong \C\times \C\times \C\times \C\times \mathrm{Mat}_{2}(\C)\]
and by \cite[Example (7.39)]{CR81}
\[\Q[D_4]\cong \Q\times \Q\times \Q\times \Q\times \mathrm{Mat}_{2}(\Q).\]
Thus $\Q[D_4]$ is an absolutely semisimple Hopf form of itself.  By Lemma \ref{aut}(ii), $\Aut(D_4)=D_4$ and so
\[\Theta(\mathrm{Map}(D_4,\Q))=\Q[D_4].\]
\end{example}

\subsubsection{The case $N=Q_8$}  Let $N=Q_8$ denote the quaternion group
\[Q_8=\{\pm1,\pm i,\pm j,\pm k\},\]
with $ij=k$, $ji=-k$, $jk=i$, $kj=-i$, $ki=j$, $ik=-j$.  We have
\[\C[Q_8]\cong \C\times \C\times \C\times \C\times \mathrm{Mat}_{2}(\C),\]
yet
\[\Q[Q_8]\cong \Q\times \Q\times \Q\times \Q\times {\mathbb H},\]
where $\mathbb H$ is the rational quaternions.  Thus $\Q[Q_8]$ is not an absolutely semisimple form of itself.   
Does $\Q[Q_8]$ admit an absolutely semisimple Hopf form?  The answer is ``yes" as we will see below.   
The solution is due to C. Greither \cite[\S 3]{CGKKKTU20}.   

We first identify a collection of Hopf forms of $\Q[Q_8]$.   Let $\zeta=\zeta_4$.  Then $\Q(\zeta)$ is a Galois extension of $\Q$ with group $C_2=\{1,g\}$.  Let $\Aut_{\scriptsize Hopf}(\Q[Q_8]))$ denote the group of automorphisms of $\Q[Q_8]$ in the category of 
$\Q$-Hopf algebras.  
By \cite[\S 17.6, Theorem]{Wa79}, the $\Q(\zeta)$-forms of $\Q[Q_8]$ are classified by the 
$1$st cohomology set
\[H^1(\Q(\zeta)/\Q,\Aut_{\scriptsize Hopf}(\Q[Q_8])).\]
Translating to $C_2$-actions, we obtain
\[H^1(\Q(\zeta)/\Q,\Aut_{\scriptsize Hopf}(\Q[Q_8]))=H^1(C_2,\Aut_{\scriptsize Hopf}(\Q(\zeta)\otimes_\Q \Q[Q_8])).\]
by \cite[\S 17.7, Theorem]{Wa79}.
Since \[\Aut_{\scriptsize Hopf}(\Q(\zeta)\otimes_\Q \Q[Q_8])=
\Aut_{\scriptsize \Q(\zeta)\hbox{-}Hopf}(\Q(\zeta)[Q_8])=\Aut(Q_8),\]
the $\Q(\zeta)$-Hopf forms of $\Q[Q_8]$ are given by the cohomology set
\[H^1(C_2,\Aut(Q_8)).\]
The action of $C_2$ on $\Aut(Q_8)$ is trivial and so a $1$-cocycle in $H^1(C_2,\Aut(Q_8))$ is a homomorphism
\[\varrho: C_2\rightarrow \Aut(Q_8).\]
Given a homomorphism $\varrho\in H^1(C_2,\Aut(Q_8))$, the corresponding $\Q(\zeta)$-Hopf form is the fixed ring
\[H(\varrho) = (\Q(\zeta)[Q_8])^{C_2},\]
where $g\in C_2$ acts on $\Q(\zeta)$ as an element of the Galois group and on $Q_8$ through $\varrho(g)$. 

Let $\theta: C_2\rightarrow \Aut(Q_8)$ be the homomorphism with $\theta(1)=1$ and where $\theta(g)$ is conjugation by $k$.  
  
\begin{proposition} [Greither]  Let $H(\theta)$ be the Hopf form of $\Q[Q_8]$ corresponding to $\theta$.  Then the Wedderburn-Artin decomposition of $H(\theta)$ is
\[H(\theta)=\Q\times \Q\times \Q\times \Q\times \mathrm{Mat}_2(\Q).\]
Thus $H(\theta)$ is the absolutely semisimple Hopf form of $\Q[Q_8]$.
\end{proposition}

\begin{proof}  We compute the fixed ring 
\[H(\theta) = (\Q(\zeta)[Q_8])^{C_2}.\]
Since the abelian part of $H(\theta)$ is equal to that of $\Q[Q_8]$, we have
\[H(\theta) = \Q\times \Q\times \Q\times \Q\times (\Q(\zeta)\otimes_\Q {\mathbb H})^{C_2}.\]
Writing
\[{\mathbb H} = (-1,-1)_\Q = \Q\oplus \Q u\oplus \Q v\oplus \Q\oplus w\]
with $u^2=-1$, $v^2=-1$, $uv=w$, $vu=-w$, 
we have that $g\in C_2$ acts on $\Q(\zeta)$ as an element of the Galois group and on ${\mathbb H}$ through conjugation by $w$. 

We claim that  
\[(\Q(\zeta)\otimes_\Q {\mathbb H})^{C_2}\cong (1,1)_\Q\cong \mathrm{Mat}_2(\Q).\] 
To this end, let
\[x=(a_0+a_1\zeta)1+(b_0+b_1\zeta)u+(c_0+c_1\zeta)v+(d_0+d_1\zeta)w\]
be an element in 
\[\Q(\zeta)\oplus \Q(\zeta) u\oplus \Q(\zeta) v\oplus \Q(\zeta) w=\Q(\zeta)\otimes_\Q {\mathbb H}\]
for $a_0,a_1,b_0,b_1,c_0,c_1,d_0,d_1\in \Q$. 
Then
\begin{eqnarray*}
g(x) & =  &  g((a_0+a_1\zeta)1+(b_0+b_1\zeta)u+(c_0+c_1\zeta)v+(d_0+d_1\zeta)w) \\
& = &  (a_0-a_1\zeta)1+(b_0-b_1\zeta)wuw^{-1}+(c_0-c_1\zeta)wvw^{-1}+(d_0-d_1\zeta)www^{-1} \\
& = &  (a_0-a_1\zeta)1+(b_0-b_1\zeta)(-u)+(c_0-c_1\zeta)(-v)+(d_0-d_1\zeta)w \\
& = &  (a_0-a_1\zeta)1+(-b_0+b_1\zeta)u+(-c_0+c_1\zeta)v+(d_0-d_1\zeta)w. \\
\end{eqnarray*}
Thus $g(x)=x$ if and only if $a_1=b_0=c_0=d_1=0$.   It follows that 
\[B=\{1,\zeta v,\zeta u,w\}\]
is a $\Q$-basis for $(\Q(\zeta)\otimes_\Q {\mathbb H})^{C_2}$.  Now,
$(\zeta v)^2=1$, $(\zeta u)^2= 1$, $(\zeta v)(\zeta u) = -1vu = w$, $(\zeta u)(\zeta v) = -1uv = -w$, and thus
$B$ is a $\Q$-basis for the quaternion algebra $(1,1)_\Q$, which by  \cite[Theorem 4.3]{Co19} is isomorphic to 
$\mathrm{Mat}_2(\Q)$.  Consequently,
\[H(\theta)=\Q\times \Q\times \Q\times \Q\times \mathrm{Mat}_2(\Q)\]
as claimed.

One could also conclude that 
\[(\Q(\zeta)\otimes_\Q {\mathbb H})^{C_2}\cong \mathrm{Mat}_2(\Q)\] 
by noting that 
$\zeta u-w$ is a non-zero nilpotent element of $(\Q(\zeta)\otimes_\Q {\mathbb H})^{C_2}$ of index $2$. 

\end{proof}

By a standard result $\Aut(Q_8)=S_4$ and so by Theorem \ref{HP} there is an $S_4$-Galois extension $L/\Q$ for which 
\[\Theta(L)=H(\theta).\]   
We compute $L/\Q$ as follows.  Let $U=\theta(C_2)\cong C_2$.  Then $\theta(C_2)$ is a subgroup of $S_4$
of index $[S_4:\theta(C_2)]=12$.  Thus by Theorem \ref{galois}, there exists an $S_4$-Galois extension of $\Q$
\[L=\underbrace{\Q(\zeta)\times \Q(\zeta)\times \cdots \times \Q(\zeta)}_{12}.\]
Moreover
\[\Theta(L)=(L[Q_8])^{S_4}=(\Q(\zeta)[Q_8])^{C_2}=H(\theta).\] 

\begin{remark}  As C. Greither has noted there exist finite groups $N$ for which $\Q[N]$ has no absolutely semisimple  Hopf forms.  For instance, take $N$ to be any non-abelian group of odd order $p^3$ and exponent $p^2$ \cite[\S 3]{CGKKKTU20}.
\end{remark}

\section{Connection to Hopf-Galois Theory}

There is a natural application of the map $\Theta$ to Hopf-Galois theory.  

\subsection{Review of Greither-Pareigis theory}

Let $E/K$ be a Galois extension with group $G$.  Let $H$ be a finite dimensional, cocommutative $K$-Hopf algebra with comultiplication 
$\Delta: H\rightarrow H\otimes_R H$, counit $\varepsilon: H\rightarrow K$, and coinverse $S: H\rightarrow H$.  Suppose there is a $K$-linear action $\cdot$ of $H$ on $E$ that satisfies
\[h\cdot(xy) = \sum_{(h)} (h_{(1)}\cdot x)(h_{(2)}\cdot y),\quad h\cdot 1 =\varepsilon(h)1\]
for all $h\in H,\;x,y\in E$, where $\Delta(h)=\sum_{(h)}h_{(1)}\otimes h_{(2)}$ is Sweedler notation. Suppose also that the $K$-linear map 
\[j: E\otimes_K H \to \End_K(E),\; j(x\otimes h)(y)=x(h\cdot y)\] 
is an isomorphism of vector spaces over $K$. Then $H$ together with this action, denoted as $(H,\cdot)$, provides a {\em Hopf-Galois structure} on 
$E/K$.   Two Hopf-Galois structures $(H_1,\cdot_1)$, $(H_2,\cdot_2)$ on $E/K$ are {\em isomorphic} if there is
a Hopf algebra  isomorphism $f: H_1\rightarrow H_2$ for which $h\cdot_1 x = f(h)\cdot_2 x$ for all $x\in E$, 
$h\in H$  (see \cite[Introduction]{CRV15}).

C. Greither and B. Pareigis \cite{GP87} have given a complete classification of Hopf-Galois structures up to isomorphism.   Denote by $\Perm(G)$ the group of permutations of $G$.  A subgroup $N\le \Perm(G)$ is {\em regular} if $|N|=|G|$ and $\eta(g)\ne g$ for all 
$\eta\ne 1_N$, $g\in G$.   Let $\lambda: G\rightarrow \Perm(G)$, $\lambda(g)(h)=gh$, denote the left regular representation.  
A subgroup $N\le \Perm(G)$ is {\em normalized} by $\lambda(G)\le \Perm(G)$ if $\lambda(G)$ is contained in the 
normalizer of $N$ in $\Perm(G)$.

\begin{theorem} [Greither and Pareigis]\label{GP} Let $E/K$ be a Galois extension with group $G$.  
There is a one-to-one correspondence between isomorphism classes of Hopf Galois structures on $E/K$ and regular subgroups of 
$\mathrm{Perm}(G)$ that are normalized by $\lambda(G)$.
\end{theorem}

One direction of this correspondence works by Galois descent:\ Let $N$ be a regular subgroup of $\Perm(G)$ normalized by $\lambda(G)$.  
Then $G$ acts on the group algebra $E[N]$ through the Galois action on $E$ and conjugation by $\lambda(G)$ on $N$, i.e., 
\[g(x\eta) = g(x)(\lambda(g)\eta\lambda(g^{-1})), g\in G, \;x\in E,\; \eta\in N.\]
We denote the conjugation action of $\lambda(g)\in\lambda(G)$ on $\eta\in N$ by $^g\eta$.  Let $H$ denote the fixed ring
\[(E[N])^G=\{x\in E[N]:\ g(x)=x, \forall g\in G\}.\]
Then $H$ is an $n$-dimensional $E$-Hopf algebra, $n=[E:K]$, and $E/K$ admits the Hopf Galois structure 
$(H,\cdot)$ \cite[p. 248, proof of 3.1 (b)$\Longrightarrow$ (a)]{GP87}, \cite[Theorem 6.8, pp. 52-54]{Ch00}.  
The action of $H$ on $E/K$ is given as 
\[\Big(\sum_{\eta\in N} r_{\eta} \eta \Big)\cdot x = \sum_{\eta\in N}r_{\eta} \eta^{-1}[1_G](x),\] 
see \cite[Proposition 1]{Ch11}.   By \cite[p. 249, proof of 3.1, (a) $\Longrightarrow$  (b)]{GP87},
\[E\otimes_K H \cong E\otimes_K K[N]\cong E[N],\]
as $E$-Hopf algebras, so $H$ is an $E$-form of $K[N]$.  

If $N$ is isomorphic to the abstract group $N'$, then we say 
that the Hopf-Galois structure $(H,\cdot)$ on $E/K$ is of {\em type} $N'$.  

\subsection{Connection to the map $\Theta$}

If $(H,\cdot)$ is a Hopf-Galois structure on $E/K$ of type $N$, then the Hopf algebra $H$ is a Hopf form of $K[N]$.  
Thus $H$ can be recovered via Theorem \ref{HP}.  In other words, with $F=\Aut(N)$, there is an $F$-Galois extension $L$ of $K$ with 
\[\Theta(L)=H=(L[N])^F.\]   

%As we have noted (Remark \ref{inverse}), it is not clear how to compute 
%the required $L$; the inverse map $\Theta^{-1}: {\mathcal Hopf}(K[N])\rightarrow {\mathcal Gal}(K,F)$ is not given explictly.

We seek a method to construct $L$.  We first prove a lemma.

\begin{lemma}  \label{normal} Let $E/K$ be a Galois extension with group $G$.  Let $(H,\cdot)$ be a Hopf-Galois structure 
corresponding to regular subgroup $N$.   Let
\[W=\{g\in \lambda(G):\ ^g\eta = \eta,\forall \eta\in N\}.\]
Then $W$ is a normal subgroup of $\lambda(G)$.
\end{lemma}
 
\begin{proof}  
Certainly $W\le \lambda(G)$.  Let $N^{opp}$ denote the centralizer of $N$ in $\Perm(G)$ and note that
\[W=\lambda(G)\cap N^{opp}.\]
The normalizer of $N$ in 
$\Perm(G)$ equals the normalizer of $N^{opp}$ in $\Perm(G)$.  Therefore since $\lambda(G)$ normalizes $N$, $\lambda(G)$ also normalizes $N^{opp}$.   Moreover, $\lambda(G)$ normalizes itself and so $\lambda(G)$ normalizes $W$, i.e., 
$W\triangleleft \lambda(G)$.

\end{proof}

Th quotient group $\lambda(G)/W$ can be viewed as a subgroup of $F=\Aut(N)$. 

\begin{proposition} \label{cond} Let $E/K$ be a Galois extension with group $G$.  Let $(H,\cdot)$ be a Hopf-Galois structure 
corresponding to regular subgroup $N$.   Let $F=\mathrm{Aut}(N)$, let
\[W=\{g\in \lambda(G):\ ^g\eta = \eta,\forall \eta\in N\},\]
and let $L=E^W$.   If $\lambda(G)/W\cong F$, then 
$\Theta(L)=H$.
Moreover, $L$ is the smallest extension of $K$ for which $L\otimes_K H\cong L[N]$. 
\end{proposition}

\begin{proof}   By the Fundamental theorem of Galois theory,
$L=E^W$ is Galois with group $F\cong \lambda(G)/W$, so $L$ is an $F$-Galois extension.  Now, 
\[H=(E[N])^{G} = (L[N])^F,\]
and so, $\Theta(L)=H$.  The second statement follows from \cite[Corollary 3.2]{GP87}.
\end{proof}

\begin{example}\label{biquad}
We consider the splitting field $E$ of the polynomial $x^4-10x^2+1$ over $\Q$.  One has $E=\Q(\sqrt {2}, \sqrt 3)$; 
$E/\Q$ is Galois with group 
$C_2\times C_2=\{1,\sigma,\tau,\tau\sigma\}$  with Galois action
\[\sigma(\sqrt {2})=\sqrt {2},\quad \sigma(\sqrt 3)=-\sqrt 3,\quad \tau(\sqrt {2})=-\sqrt {2},\quad \tau(\sqrt 3)=\sqrt 3.\]
By \cite{By02}, there are three Hopf-Galois structures on $E/\Q$ of type $C_4$, each of which is determined by a regular subgroup $N\cong C_4$ normalized by $\lambda(C_2\times C_2)$.    One such $N$ is given as
\[N=\{(1),(1,3,2,4),(1,2)(3,4),(1,4,2,3)\},\]
where $1:=1$, $2:=\sigma$, $3:=\tau$, $4:=\tau\sigma$, and 
\[\lambda(C_2\times C_2)=\{(1),(1,2)(3,4),(1,3)(2,4),(1,4)(2,3)\}.\]
$N$ is a regular subgroup of $\Perm(C_2\times C_2)$ normalized by $\lambda(C_2\times C_2)$ with
$N\cong C_4$. 

Let $(H,\cdot )$ be the corresponding Hopf-Galois extension with $H=(E[N])^G$.  As one can check 
\[W=\{g\in \lambda(C_2\times C_2):\ ^g\eta =\eta, \forall \eta\in N\}=\{(1),(1,2)(3,4)\}=\{1,\sigma\}.\]   
We have $\lambda(C_2\times C_2)/W\cong F=\mathrm{Aut}(C_4)\cong C_2$, and the fixed field $L=E^W=\Q(\sqrt{2})$ is an $F$-Galois extension of $\Q$.  So by Proposition \ref{cond}
\[\Theta(L)=H.\]  
\end{example}

\begin{example}  \label{S3}
Let $E=\Q(\zeta_3,\root 3\of 2)$ be the splitting field of $x^3-2$ over $\Q$, so that $E/\Q$ is Galois with group $S_3$. 
By \cite[Lemma 1]{KKTU19}, there are three Hopf-Galois
structures on $E/\Q$ of type $C_6$, all of which have isomorphic Hopf algebras $H$ by \cite[Proposition 3]{KKTU19}.  We have
\[W=\{g\in \lambda(S_3):\ ^g\eta =\eta, \forall \eta\in C_6\}\cong C_3.\]   
Thus $\lambda(S_3)/W\cong C_2=\Aut(C_6)$ and 
$E^W=\Q(\zeta_3)$.  Consequently, 
\[\Theta(\Q(\zeta_3))=H\]
by Proposition \ref{cond}.  In this case $H=(\Q[C_6])^*$ is the absolutely semisimple Hopf form of $\Q[C_6]$.  
\end{example}

\begin{example}  \label{complete}  Let $E/K$ be a Galois extension with group $G$ where $G$ is a non-abelian complete group (i.e., $G$ has trivial center and $G\cong \Aut(G)$).  For instance, $G=S_n$, $n>6$, is a non-abelian complete group.

The subgroup $N=\lambda(G)$ is a regular subgroup of $\Perm(G)$ normalized by itself and corresponds to the 
canonical non-classical  Hopf-Galois structure with Hopf algebra $H_\lambda$.   In this case $W$ is trivial since the center of $G$ is trivial, and so $E^W=E$ and 
\[\lambda(G)/W\cong \lambda(G)\cong \Aut(N)=F.\]
Thus by Proposition \ref{cond}
\[\Theta(E)=H_\lambda.\]

\end{example}

\subsection{When $\lambda(G)/W\not\cong F$}

In Example \ref{biquad}, Example \ref{S3}, and Example \ref{complete} we have $\lambda(G)/W\cong F$ and 
Proposition \ref{cond} can be applied to find $L$ so that $\Theta(L)=H$.  
We next consider the case where $\lambda(G)/W$ is a proper subgroup of $F$. 

Let $E/\Q$ be Galois with quaternion group $Q_8=\{\pm1,\pm i,\pm j,\pm k\}$.
Then $\{\pm1\}$ is the unique order $2$ subgroup of $Q_8$.  Let $K=E^{\{\pm1\}}$.  
Since $\{\pm1\}$ is normal in $Q_8$, $K/\Q$ is Galois with group 
\[Q_8/\{\pm1\}=\{\overline 1,
\overline i,\overline j,\overline k\}\cong C_2\times C_2.\]  
Thus $K/\Q$ is the unique biquadratic subfield of $E$.  For $s,t\in \{i,j,k\}$ with $s\not = t$, 
there exist elements $\alpha$, $\beta$ in $K$ satisfying 
$\alpha^2\in \Q$, $\beta^2\in \Q$ with $K=\Q(\alpha,\beta)$.   We have
\[s(\alpha)=\alpha,\quad  t(\alpha)=-\alpha,\quad  s(\beta)=-\beta,\quad  t(\beta)=\beta.\] 
Note that 
\[E^{\langle t\rangle} = \Q(\beta).\]

S. Taylor and P. J. Truman \cite[Lemma 2.3]{TT19}
have shown that there are $6$ Hopf-Galois structures of type $N=C_8$ on $E/\Q$.  The $6$ regular subgroups are given 
as the $6$ subgroups 
\[\{C_{s,t}\mid s,t\in \{i,j,k\}, s\not = t\},\]
where
 \[C_{s,t}=\langle\eta_{s,t}\rangle\cong C_8.\]
Here the generator $\eta_{s,t}$ denotes the permutation in cycle notation
\[\eta_{s,t} = (1,s,t,t^{-1}s^{-1},s^2,s^{-1},t^{-1},st).\]
The action of $G$ on $C_{s,t}$ is given as
\[^s\eta_{s,t} = \lambda(s)\eta_{s,t}\lambda(s^{-1})=\eta_{s,t}^3,\]
\[^t\eta_{s,t} = \lambda(t)\eta_{s,t}\lambda(t^{-1}) = \eta_{s,t}.\]
(See \cite[proof of Lemma 2.3]{TT19}.)

Let $s,t\in \{i,j,k\}$, $s\not = t$.  Then by Theorem \ref{GP} the fixed ring
\[H_{s,t} = (E[C_{s,t}])^G\]
is the Hopf algebra attached to the Hopf-Galois structure on $E/\Q$;
$H_{s,t}$ is an $E$-form of $\Q[C_8]$.  Note that
\[W=\{g\in \lambda(Q_8):\ ^g\eta = \eta,\forall \eta\in C_{s,t}\}=\langle t\rangle,\]
with $\lambda(Q_8)/W\cong C_2$.  Thus $\lambda(Q_8)/W$ is a proper subgroup of $F=\Aut(C_8)\cong C_2\times C_2$.

\subsection{Finding $L$ so that $\Theta(L)=H_{s,t}$}

We want recover $H_{s,t}$ using Theorem \ref{HP}, i.e., with $F=\Aut(C_8)$, we want to find an 
$F$-Galois extension $L$ for which 
\[\Theta(L)=(L[C_8])^F=H_{s,t}.\]

We have 
\[F=\Z_8^* = \{1,3,5,7\}\cong C_2\times C_2.\]
Let $U=\langle 3\rangle\le F$.  We identify $3$ with the element $s\in \{i,j,k\}$.  Then $\Q(\beta)$ is a Galois extension of 
$\Q$ with group $U$; $3(\beta)=-\beta$.   (Recall that $E^{\langle t\rangle} = \Q(\beta)$.)   Let $\{1,5\}$ be a left transversal for $U$ in $F$, i.e., the left cosets of $U$ in $F$
are 
\[\langle 3\rangle,\quad 5\langle 3\rangle.\]

Let
\[L=\Q(\beta)f_1\oplus \Q(\beta)f_2\]
where $f_1$, $f_2$ are mutually orthogonal idempotents.  A typical element of $L$ can be written as
\[x= (a_0+a_1\beta)f_1+(b_0+b_1\beta)f_2,\]
where $a_0,a_1,b_0,b_1\in \Q$.   Following \cite[proof of Theorem 4.2]{Pa90}, $L$ is an $F$-Galois extension of $\Q$ with action of $F$ on $L$ given as
\[1\cdot x = x,\]
\[3\cdot x = (a_0-a_1\beta)f_1+(b_0-b_1\beta)f_2,\]
\[5\cdot x = (b_0+b_1\beta)f_1+(a_0+a_1\beta)f_2,\]
\[7\cdot x = (b_0-b_1\beta)f_1+(a_0-a_1\beta)f_2.\]

\begin{proposition}  \label{Q8} For $s,t\in \{i,j,k\}$, $s\not = t$, let $H_{s,t}$ be the Hopf algebra for the Hopf-Galois structure on 
$E/\Q$ corresponding to regular subgroup $C_{s,t}$.  Let $L$ be as above.  Then
\[\Theta(L)=(L[C_8])^F=H_{s,t}.\]
\end{proposition}

\begin{proof}    

We compute a $\Q$-basis for the fixed ring directly.  Let $\eta=\eta_{s,t}$ and let 

 \[\sum_{m=0}^7 ((a_{m,0}+a_{m,1}\beta)f_1+(b_{m,0}+b_{m,1}\beta)f_2)\eta^m\]
be a typical element of $L[C_{s,t}]$.   Then 

\[g\left (\sum_{m=0}^7 ((a_{m,0}+a_{m,1}\beta)f_1+(b_{m,0}+b_{m,1}\beta)f_2)\eta^m\right )\]
\[=\sum_{m=0}^7 ((a_{m,0}+a_{m,1}\beta)f_1+(b_{m,0}+b_{m,1}\beta)f_2)\eta^m\]
for all $g\in F=C_2\times C_2$ if and only if

\[a_{0,0}=b_{0,0},\quad a_{0,1}=b_{0,1}=0,\]

\[a_{1,0}=a_{3,0}=b_{5,0}=b_{7,0},\quad a_{1,1} = -a_{3,1} = b_{5,1} = -b_{7,1},\]

\[b_{1,0}=b_{3,0}=a_{5,0}=a_{7,0},\quad b_{1,1} = -b_{3,1} = a_{5,1} = -a_{7,1},\]

\[a_{2,0}=a_{6,0}=b_{2,0}=b_{6,0},\quad a_{2,1} = -a_{6,1} = b_{2,1} = -b_{6,1},\]

\[a_{4,0}=b_{4,0},\quad a_{4,1}=b_{4,1}=0.\]
It follows that a $\Q$-basis for $(L[C_8])^F$ is 

\[1,\quad \eta^4,\quad \eta^2+\eta^6,\quad \beta(\eta^2-\eta^6),\]

\[f_1(\eta+\eta^3)+f_2(\eta^5+\eta^7),\quad f_2(\eta+\eta^3)+f_1(\eta^5+\eta^7),\]

\[\beta(f_1(\eta-\eta^3)+f_2(\eta^5-\eta^7)),\quad \beta(f_2(\eta-\eta^3)+f_1(\eta^5-\eta^7)).\]
Thus a $\Q$-basis for $(L[C_8])^F$ can be written as

\[{1\over 8}\left (1+\eta+\eta^2+\eta^3+\eta^4+\eta^5+\eta^6+\eta^7\right ),\]

\[{1\over 8}\left (1-\eta+\eta^2-\eta^3+\eta^4-\eta^5+\eta^6-\eta^7\right ),\]

\[{1\over 4}\left (1-\eta^2+\eta^4-\eta^6\right ),\quad {1\over 4}\beta\left ((\eta-\eta^3)+(\eta^5-\eta^7)\right ),\]

\[{1\over 2}(1-\eta^4),\quad {1\over 2}\beta(\eta^2-\eta^6),\]
together with two other basis elements

\begin{eqnarray*}
&&{1\over 2}(f_2(\eta+\eta^3)+f_1(\eta^5+\eta^7))-{1\over 2}(f_1(\eta+\eta^3)+f_2(\eta^5+\eta^7)) \\
&&\quad\ =\ {1\over 2}(f_2-f_1)((\eta^3-\eta^7)+(\eta-\eta^5)),
\end{eqnarray*}
and 

\begin{eqnarray*}
&&{1\over 2}\beta (f_1(\eta-\eta^3)+f_2(\eta^5-\eta^7))-{1\over 2}\beta (f_2(\eta-\eta^3)+f_1(\eta^5-\eta^7)) \\
&&\quad\ =\ {1\over 2}\beta (f_2-f_1)((\eta^3-\eta^7)-(\eta-\eta^5)).
\end{eqnarray*}
Note that 

\[(f_2-f_1)^2=f_2+f_1=1.\]

Considering the $\Q$-basis for $H_{s,t}$ given in \cite[page 231]{TT19},
we see that there is an isomorphism of $\Q$-algebras 
\[\psi: (L[C_8])^F\rightarrow H_{s,t}\]
defined by
\[{1\over 8}\left (1+\eta+\eta^2+\eta^3+\eta^4+\eta^5+\eta^6+\eta^7\right )\mapsto {1\over 8}\left (1+\eta+\eta^2+\eta^3+\eta^4+\eta^5+\eta^6+\eta^7\right ),\]

\[{1\over 8}\left (1-\eta+\eta^2-\eta^3+\eta^4-\eta^5+\eta^6-\eta^7\right )\mapsto {1\over 8}\left (1-\eta+\eta^2-\eta^3+\eta^4-\eta^5+\eta^6-\eta^7\right ),\]

\[{1\over 4}\left (1-\eta^2+\eta^4-\eta^6\right )\mapsto {1\over 4}\left (1-\eta^2+\eta^4-\eta^6\right ),\]

\[{1\over 4}\beta\left ((\eta-\eta^3)+(\eta^5-\eta^7)\right )\mapsto {1\over 4}\beta\left ((\eta-\eta^3)+(\eta^5-\eta^7)\right ),\]

\[{1\over 2}(1-\eta^4)\mapsto {1\over 2}(1-\eta^4),\]

\[{1\over 2}\beta(\eta^2-\eta^6)\mapsto {1\over 2}\beta(\eta^2-\eta^6),\]

\[{1\over 2}(f_2-f_1)((\eta^3-\eta^7)+(\eta-\eta^5))\mapsto {1\over 2}((\eta^3-\eta^7)+(\eta-\eta^5)),\]

\[{1\over 2}\beta (f_2-f_1)((\eta^3-\eta^7)-(\eta-\eta^5))\mapsto {1\over 2}\beta ((\eta^3-\eta^7)-(\eta-\eta^5)).\]

Moreover, since there is an embedding of $L$-Hopf algebras
\[L[(f_2-f_1)(1-\eta^4)]\rightarrow L[1-\eta^4]\cong L[\langle \eta^4\rangle]\cong L[C_2]\]
given as $(f_2-f_1)(1-\eta^4)\mapsto 1-\eta^4$,
we conclude that the $\Q$-algebra isomorphism $\psi$ is also an isomorphism of $\Q$-Hopf algebras.

\end{proof}

As a byproduct of Proposition \ref{Q8}, we can recover \cite[Lemma 3.3]{TT19}.

\begin{proposition}  Let $s,s',t,t'\in \{i,j,k\}$ with $s\not = t$, $s'\not = t'$.  Then $H_{s,t}\cong H_{s',t'}$ as Hopf algebras
if and only if $t=t'$. 
\end{proposition}

\begin{proof}   If  $t=t'$ then the fixed fields $E^{\langle t\rangle}$ and $E^{\langle t'\rangle}$ are both equal to 
$\Q(\beta)$.  Let 
\[L=\Q(\beta)f_1\oplus \Q(\beta)f_2.\]  
Then $\Theta(L)=H_{s,t}=H_{s',t'}$ so that $H_{s,t}\cong H_{s',t'}$.  

For the converse, suppose that $H_{s,t}\cong H_{s',t'}$.  Then there is an $F$-Galois extension $L$ of $\Q$ with 
$\Theta(L)=H_{s,t}$ and an $F$-Galois extension $L'$ of $\Q$ with $\Theta(L')=H_{s',t'}$.   Necessarily,
$L\cong \Q(\beta)f_1\oplus \Q(\beta)f_2$ where $E^{\langle t\rangle} = \Q(\beta)$ and 
$L'\cong \Q(\gamma)f_1\oplus \Q(\gamma)f_2$ where $E^{\langle t'\rangle} = \Q(\gamma)$.   
Now $H_{s,t}\cong H_{s',t'}$ implies that $L\cong L'$, which says that $\Q(\beta)\cong \Q(\gamma)$, hence $t=t'$.   
\end{proof}

\end{document}